\documentclass[a4paper, 12pt]{article}

\usepackage[T1]{fontenc} 
\usepackage[utf8]{inputenc} 
\usepackage{fullpage} 
\usepackage{amsmath} 
\usepackage{amsfonts} 
\usepackage{amssymb} 
\usepackage{amsthm} 
\usepackage{mathrsfs}
\usepackage{bbm}

\usepackage{tikz}

\usepackage{hyperref}

\newcommand{\N}{\mathbb{N}} 
\newcommand{\Q}{\mathbb{Q}} 
\newcommand{\R}{\mathbb{R}} 

\newcommand{\Prob}{\mathbb P}
\newcommand{\Probb}{\mathbb Q}
\newcommand{\E}{\mathbb E}

\newcommand{\I}{\mathbbm{1}}
\newcommand{\eps}{\varepsilon}
\newcommand{\dx}{\,\mathrm{d}}
\newcommand{\id}{\overset{\mathrm{d}} {=}}

\newcommand{\convd}{\overset{\mathrm{d}} {\rightarrow}}
\newcommand{\convfdd}{\overset{\mathrm{fdd}} {\rightarrow}}

\theoremstyle{plain} 
\newtheorem{theorem}{Theorem}[section] 
\newtheorem{lemma}[theorem]{Lemma} 

\newtheorem{proposition}[theorem]{Proposition}
\theoremstyle{definition} 

\theoremstyle{remark} 
\newtheorem*{remark}{Remark}

\begin{document}
\title{Penalizing fractional Brownian motion for being negative}
\author{Frank Aurzada\footnote{Technical University of Darmstadt, Schloßgartenstraße 7, 64289 Darmstadt, Germany. E-mail: aurzada@mathematik.tu-darmstadt.de, buck@mathematik.tu-darmstadt.de, kilian@mathematik.tu-darmstadt.de} \and Micha Buck\footnotemark[1] \and Martin Kilian\footnotemark[1]}
\date{\today}
\maketitle

\begin{abstract}
We study a modification of the fractional analogue of the Brownian meander, which is Brownian motion conditioned to be positive on the time interval ${[0,1]}.$ More precisely, we determine the weak limit of a fractional Brownian motion which is penalized -- instead of being killed -- when leaving the positive half-axis. 
\\In the Brownian case, we give a representation of the limiting process in terms of an explicit SDE and compare it to the SDE fulfilled by the Brownian meander. 
\end{abstract}

\section{Introduction}
The aim of this paper is to make a first contribution to the rigorous study of fractional Brownian motion conditioned to be positive. Defining such a new object and studying its properties is a goal that is formulated in the theoretical physics literature (e.g.~\cite{Zoia2009} and \cite{Majumdar2010}) and is relevant for many interesting physical systems with long-range dependence. It is also a mathematically interesting object, as it would be one of the few occasions where one can define a process conditioned on an event of probability zero outside the realm of Markov processes.

Let $(B_H(t))_{t \ge 0}$ be a standard fractional Brownian motion with Hurst parameter $H \in {(0,1)},$ i.e.~an a.s.~continuous, centred Gaussian process with covariance function $(s,t) \mapsto \frac {1} {2} \left(t^{2H} + s^{2H} - |t-s|^{2H}\right)$ on some probability space $(\Omega, \mathscr{F}, \Prob).$ 

Analogously to the well-known Brownian case $H=1/2,$ there are two approaches of conditioning to be positive: One may consider the sequence of measures
\begin{align}
&\Prob{\left(\left.(B_H(t))_{t \in {[0,1]}} \in \,\cdot ~ \right| \ B_H(s) \ge -T^{-H} \, \forall s \in {[0,1]}\right)} \notag
\\&=\E{\left[\I_{ (T^{-H} B_H(Tt))_{t\in[0,1]} \in \,\cdot ~} \cdot\frac{\I_{B_H(s) \ge -1 \, \forall s \in {[0,T]}}}{\Prob(B_H(s) \ge -1 \, \forall s \in {[0,T]})}\right]}, \ T \to \infty, \label{eq:indicator}
\end{align}
on $C{[0,1]},$ which we want to focus on, or alternatively, for every $t>0,$ the sequence
\begin{equation}
\E{\left[\I_{ (B_H(s))_{s \in {[0,t]}} \in \,\cdot ~} \cdot \frac{\I_{B_H(s) \ge -1 \, \forall s \in {[0,T]}}}{\Prob(B_H(s) \ge -1 \, \forall s \in {[0,T]})}\right]}, \ T \to \infty, \label{eq:bessel}
\end{equation}
on $C{[0,t]}.$ The equality in \eqref{eq:indicator} is due to the self-similarity of $B_H.$ 

The goal of this paper is to determine the weak limit of the following modification of the sequence in \eqref{eq:indicator}:
\begin{equation} \label{eq:indicatorsmooth}
\E{\left[\I_{ (T^{-H} B_H(T t))_{t\in[0,1]} \in \,\cdot ~} \cdot\frac{ \left( \int_0^T \exp(-B_H(s)) \dx s\right)^{-1}}{I(T)}\right]}, \ T \to \infty,
\end{equation}
where
\begin{equation*}
I(T):=\E{\left[\left( \int_0^T \exp(-B_H(s)) \dx s\right)^{-1} \right]}.
\end{equation*}
This is motivated as follows: In order to study the normalization sequence $\Prob(B_H(s) \ge -1 \, \forall s \in {[0,T]})$ in \eqref{eq:indicator}, which is also called persistence probability, the indicator was replaced by the smoother functional $\left( \int_0^T \exp(-B_H(s)) \dx s\right)^{-1}$ by Molchan, see \cite[Statement~1]{Molchan1999}. In fact, the polynomial asymptotic rate of the persistence probability is the same as the one of $I(T).$ For the Brownian case, this result had already been contained in \cite[Section~2.2]{KawazuTanaka1993}. The technique used in \cite[Statement~1]{Molchan1999} will also be the basic guiding line in our study of the weak limit of \eqref{eq:indicatorsmooth}. 

Heuristically, there is the following connection between the persistence probability and the smoothed out counterpart $I(T)$: The typical paths of the fractional Brownian motion $B_H$ contributing to the persistence event (i.e., which satisfy $B_H(s) \ge -1 \, \forall s \in {[0,T]}$) tend to escape to $+\infty$ rather than oscillating around the origin. But these are exactly those paths for which the functional $\left( \int_0^T \exp(-B_H(s)) \dx s\right)^{-1}$ is large and which contribute to $I(T)$ the most, consequently. 

After studying the weak limit of \eqref{eq:indicatorsmooth}, we shall investigate the Brownian case in more detail and show that the resulting limiting process satisfies a concrete stochastic differential equation. For comparison, in the Brownian setup, the sequence in \eqref{eq:indicator} weakly converges to the law of the Brownian meander (see \cite{Durrett1977} for the Brownian motion or \cite{Bolthausen1976} and \cite{Iglehart1974} for the corresponding discrete analogues), whereas the weak limit of the sequence in \eqref{eq:bessel} is given by the law of the three-dimensional Bessel process, started in $1$ and translated to the origin, on ${[0,t]}$ (see e.g.~\cite[Theorem~4.18]{YorRoynette2009}). Therefore, we will compare the SDE which is fulfilled by our limiting process with the one of the Brownian meander. 

Note that the proofs of the existing weak convergence results crucially rely on the Markov property of the Brownian motion, by using e.g.~convergence of transition densities in the case of \eqref{eq:indicator} or the Doob $h$-transform approach in the case of \eqref{eq:bessel}. In the fractional setting, this obviously does not carry over due to the intrinsically non-Markovian structure of fractional Brownian motion. 

Our results are related to the study of persistence probabilities. In this context, persistence means that a stochastic process has a long positive excursion. Analysing the asymptotics of persistence probabilities attracted a great deal of attention in recent years, with applications especially in theoretical physics, see \cite{Majumdar1999} and \cite{BrayMajumdarSchehr2013}. For an overview concerning mathematical results, we refer to the survey \cite{AurzadaSimon2015}. While the main persistence result for fractional Brownian motion has been proved by Molchan already in 1999, \cite{Molchan1999}, results for the corresponding discrete analogues -- fractional sums converging in distribution to the fractional Brownian motion -- were obtained rather recently, see \cite{AurzadaGuillotinPlantardPene2018}, \cite{AurzadaBuck2018} and \cite{Lyu2019}.

This paper is organized as follows: In the next section, we first present our main result, the weak convergence of the distributions in \eqref{eq:indicatorsmooth} to the law of a fractional Brownian motion under a modified probability measure which is equivalent to $\Prob$ (Section \ref{subsec:weakconv}). Afterwards, we analyse the limiting process in the Brownian case and show that it fulfills an explicit stochastic differential equation (Section \ref{subsec:sde}). Finally, the proofs of all the results are given in Section~\ref{sec:proof}.

\section{Results}
\subsection{Weak convergence result}\label{subsec:weakconv}
We begin by stating our main result:
\begin{theorem}\label{thm:main}
Let $(B_H(t))_{t \ge 0}$ be a fractional Brownian motion with Hurst parameter $H \in {(0,1)}$ and, for every $T \ge 1,$ let $(X_{H,T}(t))_{t \in {[0,1]}}$ be a process whose distribution is given by \eqref{eq:indicatorsmooth}, i.e.
\begin{equation*}
\Prob((X_{H,T}(t))_{t \in {[0,1]}} \in \,\cdot ~)=\E{\left[\I_{ (T^{-H} B_H(T t))_{t\in[0,1]} \in \,\cdot ~} \cdot\frac{ \left( \int_0^T \exp(-B_H(s)) \dx s\right)^{-1}}{I(T)}\right]}.
\end{equation*}
Then, for $T \to \infty,$
\begin{equation*}
(X_{H,T}(t))_{t \in {[0,1]}} \convd \left(X_H(t)\right)_{t \in {[0,1]}},
\end{equation*}
on $(C{[0,1]}, \lVert \cdot \rVert_{\infty}),$ where $\left(X_H(t)\right)_{t \in {[0,1]}}$ is a process whose law is given by
\begin{equation*}
\Prob((X_H(t))_{t \in {[0,1]}} \in \,\cdot ~):=\E{\left[\mathbbm{1}_{(B_H(t))_{t \in {[0,1]}} \in \,\cdot ~} \ \cdot \frac {B_H(1) - M_H(1)} {\E[-M_H(1)]}\right]}
\end{equation*}
with $M_H(t):=\min_{s \in {[0,t]}} B_H(s), \ t \ge 0,$ being the running minimum of $B_H.$ 
\end{theorem}

Note that the distribution of the limiting process $X_H$ under $\Prob$ is equal to the law of $B_H$ under the probability measure
\begin{equation}\label{eq:measureQ}
\Probb(A):=\int_A \frac {B_H(1) - M_H(1)} {\E[-M_H(1)]} \dx \Prob, \ A \in \mathscr{F}.
\end{equation}
As one would expect, the density $\frac {\mathrm{d} \Probb} {\mathrm{d} \Prob}$ rewards the paths that tend to escape to $+\infty,$ since in this case, $B_H(1)$ becomes large and $M_H(1)$ stays near to zero. In particular, the new process $X_H$ loses the property of symmetry of the original process $B_H,$ as one can see easily that $\Prob(X_H(1) > 0) > \Prob(X_H(1) < 0).$ However, contrary to a possible limit of the distributions of the original problem in \eqref{eq:indicator}, the limit distribution $\Prob(X_H \in \,\cdot ~)$ is not concentrated on paths staying positive. 

\begin{remark}
In \cite{Molchan1999}, various other scenarios are studied that resemble the persistence
event $\{ B_H(t) \leq 1 , t \in [0,T] \}.$ It is shown that the
probabilities of the following events also have the same polynomial rate:
\begin{itemize}
\item an early position of the maximum, $\{ \operatorname{argmax}_{t\in[0,T]} B_H(t) < 1 \},$
\item an early last zero crossing, $\{ \sup \{ t \in [0,T] : B_H(t)=0 \} <
1 \},$ and
\item a small total time spent on the negative half-axis, $\{ |\{ t\in[0,T]
: B_H(t)<0 \}| < 1 \}.$
\end{itemize}
It would be very interesting to study the $T\to\infty$ limit of $B_H$
conditioned on any of the three last events, and
of course on the persistence event itself, as in \eqref{eq:indicator}. 
\end{remark}

\subsection{Explicit SDE in the Brownian case}\label{subsec:sde}

We consider the Brownian case in this subsection, that is, we consider the case $H=1/2.$ To simplify notation, we abbreviate $(X_{1/2}(t))_{t \in [0,1]}$ to $(X(t))_{t \in [0,1]}$ and set $M_X(t):= \min_{s \in [0,t]} X(s).$ 
Further, $\Phi_{1-t}$ denotes the distribution function of an $\mathcal{N}(0,1-t)$-distribu\-tion. 

\begin{proposition}
\label{prop:BrownianCase}
Let $c \colon [0,1] \times [0,\infty) \to [0,\infty)$ be defined as
\begin{equation*}
c(t,x) := \frac{2 \,\Phi_{1-t}(x)-1}{x + 2 \int_{x}^\infty (1-\Phi_{1-t}(s)) \dx s}.
\end{equation*}
Then, a Brownian motion $(\tilde{B}(t))_{t \in [0,1]}$ exists such that $(X(t))_{t \in [0,1]}$ satisfies the SDE 
\begin{equation*}
\dx X(t) = c(t,X(t)-M_X(t)) \dx t + \dx \tilde{B}(t).
\end{equation*} 
\end{proposition}

In the following, we will discuss the nature of the process $(X(t))_{t \in [0,1]}$ and we will compare it (in terms of SDEs) to the limiting processes of \eqref{eq:indicator} and \eqref{eq:bessel}, respectively. For this purpose, we recall that the limiting process in \eqref{eq:bessel} is a three-dimensional Bessel process, started in $1$ and translated to the origin, which we denote by $(X^{(be)}(t))_{t \ge 0}.$ Further, it is well-known that a Brownian motion $(\tilde{B}(t))_{t \ge 0}$ exists such that the SDE
\begin{equation}
\label{eq:SDEbessel}
\dx X^{(be)}(t) = c^{(be)}(X^{(be)}(t) + 1) \dx t + \dx \tilde{B}(t)
\end{equation}
with
\begin{equation*}
c^{(be)}(x) := c^{(be)}(t,x) := \frac{1}{x}
\end{equation*}
is satisfied, see e.g.~\cite[Proposition 3.21]{KaratzasShreve}.

In the following lemma (cf. \cite[Lemme~2]{AzemaYor1989}), we provide an SDE that is satisfied by a Brownian meander, which is the limiting process in \eqref{eq:indicator}. 
\begin{lemma}
\label{lem:SDEmeander}
Let $(X^{(me)}(t))_{t \in [0,1]}$ be a Brownian meander and $c \colon [0,1] \times (0,\infty) \to (0,\infty)$ be defined as
\begin{equation*}
c^{(me)}(t,x) := \frac {\exp\!{\left(-\frac {x^2} {2 (1-t)}\right)}} {\int_0^x \exp\!{\left(-\frac {y^2} {2 (1-t)}\right)} \dx y}.
\end{equation*}
Then, a Brownian motion $(\tilde{B}(t))_{t \in [0,1]}$ exists such that $(X^{(me)}(t))_{t \in [0,1]}$ satisfies the SDE 
\begin{equation*}
\dx X^{(me)}(t) = c^{(me)}(t,X^{(me)}(t)) \dx t + \dx \tilde{B}(t).
\end{equation*}
\end{lemma}

\begin{remark}
While preparing the first version of this paper, to our surprise and to the best of our knowledge, we could not find any reference in the literature for an SDE that is satisfied by a Brownian meander. After the submission of the first version, we became aware that this is included in \cite{AzemaYor1989} as well as in \cite[Exercise~3.6]{MansuyYor2008} as a special case of the generalized meander.

Therefore, we do not include a proof of Lemma \ref{lem:SDEmeander} (cf. arXiv:1907.07608v1 [math.PR] for an alternative proof). Nevertheless, we think it is worth pointing out that $X^{(me)}$ can also be characterized as a Doob $h$-transform of Brownian motion: Recall that the transition density of $X^{(me)}$ is given by 
\begin{equation*}
\Prob{\left(\left.X^{(me)}(t+s) \in \dx y \ \right| X^{(me)}(t) = x \right)} = \left( \varphi_s(y-x) - \varphi_s(y+x)  \right) \frac{\Phi_{1-t-s}(y)-\frac {1} {2}}{\Phi_{1-t}(x)-\frac {1} {2}} \dx y,
\end{equation*}
where $\varphi_s$ is the density of an $\mathcal{N}(0,s)$-distribution, 
see e.g.~\cite{Durrett1977}.
Now, it is just a simple observation that this transition density is obtained by the Doob $h$-transform, for the harmonic function $h(t,x) = \Prob_x( B(s)>0 \,\forall s \in [0,1-t] )=2 \Phi_{1-t}(x)-1$ of Brownian motion on $[0,\infty)$, with killing at $0.$ The last equality follows straightforwardly from the reflection principle. Similarly, the transition density of Brownian motion on $[0,\infty)$, with killing at $0,$ can be obtained from the refection principle.  The explicit expression of $h$ can be used to obtain the SDE in Lemma~\ref{lem:SDEmeander} directly, see \cite[Section~IV.39]{Rogers2000}.
\end{remark}

In the following, we compare our limiting process $X$ and the Brownian meander $X^{(me)}$ in terms of SDEs. Additionally, the limiting process $X^{(be)}$ in \eqref{eq:bessel} turns out to be useful for comparisons, but recall that the approach of conditioning in \eqref{eq:bessel} is different to the approach where one gets $X$ and $X^{(me)}$ as limiting processes. Further, we emphasize that, in the situation of Proposition \ref{prop:BrownianCase}, one would think of $c(t,x)$ as the drift away from the former minimum of the process at time $t$, whereas $c^{(me)}(t,x)$ can be thought of as the drift away from $0$ at time $t$. Similarly, $c^{(be)}(t,x)$ can be seen as the drift away from $-1$ at time $t$ of $X^{(be)}$ or, equivalently, as the drift away from $0$ at time $t$ of the three-dimensional Bessel process started at $0.$

We first examine the time dependence of the drift terms. While the drift effect of the terms $c(t,x)$ and $c^{(me)}(t,x)$ changes in time, there is no time dependence in $c^{(be)}(t,x).$ This is rather unsurprising, since the three-dimensional Bessel process can be thought of as a Brownian motion that is conditioned to stay positive for the infinite future, whereas the conditions for the processes $(X(t))_{t \in [0,1]}$ and $(X^{(me)}(t))_{t \in [0,1]},$ respectively, just take the future until time $1$ into account. The drift effect for the Brownian meander decreases in time because with progressing time, it becomes easier to stay above zero until time $1$. In contrast, the drift $c(t,x)$ increases in time. That is, it becomes more unfavourable to first attain a new minimum and then maximize the difference $X(1)-M_X(1)$. The latter quantity comes from the distribution in \eqref{eq:measureQ} that favours paths with a large difference $X(1)-M_X(1)$. Moreover, we note that $c(t,x) \to c^{(be)}(t,x)$, as $t \to 1$.

Now, we fix time $t$ and vary $x$. 
One has $c(t,x) \sim c^{(be)}(x),$ as $x \to \infty.$ Thus, for large values of $x,$ the drift that maximizes the difference $X(1)-M_X(1)$ in the case of Proposition \ref{prop:BrownianCase} and the drift that prevents the process from becoming negative in the three-dimensional Bessel case coincide. 
But we see a completely different behaviour at $0.$ The drift term vanishes when the process gets close to its former minimum, i.e. $c(t,x) \to 0,$ as $x \to 0.$ This seems to be natural, since in this case, it is not as expensive anymore to take first a new minimum and maximize the difference $X(1)-M_X(1)$ afterwards.

In contrast, we can observe sort of the opposite behaviour for the Brownian meander. Here, we have $c^{(me)}(t,x) \sim C \exp(-x^2/2(1-t))$ for some appropriate constant $C > 0,$ as $x \to \infty.$ Thus, the drift term decays much faster, which is due to the finite time horizon in which the process has to stay positive (in contrast to the infinite time horizon in the three-dimensional Bessel case). For $x \to 0,$ we have $c^{(me)}(t,x) \sim c^{(be)}(t,x) = 1/x.$ This seems to be natural again, since being close at $0,$ only the drift that pushes the process away from $0$ becomes relevant. 

\section{Proofs} \label{sec:proof}
\subsection{Proof of Theorem \texorpdfstring{\ref{thm:main}}{2.1}}
The proof is divided into two lemmata. By \cite[Theorem~8.1]{Billingsley1968}, it suffices to show that the finite-dimensional distributions of $X_{H,T}$ converge for $T \to \infty$ to those of $X_H,$ and that the family $\big(\Prob{\left(X_{H,T} \in \,\cdot ~\right)}\big)_{T \ge 1}$ of distributions on $C{[0,1]}$ is tight. 
\\In fact, both necessary ingredients can be proven by adapting Molchan's approach in the proof of \cite[Statement~1]{Molchan1999} for the calculation of the asymptotics of $I(T)$ to our slightly different problem of analysing the expectation of some indicator multiplied by the functional $\left( \int_0^T \exp(-B_H(s)) \dx s\right)^{-1}.$ 
\\We start with the convergence of the finite-dimensional distributions:
\begin{lemma}\label{lem:fdd}
In the setting of Theorem \ref{thm:main}, we have, for $T \to \infty,$
\begin{equation*}
(X_{H,T}(t))_{t \in {[0,1]}} \convfdd \left(X_H(t)\right)_{t \in {[0,1]}},
\end{equation*}
where $\convfdd$ means convergence of the finite-dimensional distributions. 
\end{lemma}
\begin{proof}
Let $d \in \N, \ t_1,\dots,t_d \in {[0,1]}$ and $B \in \mathscr{B}{\left(\R^d\right)}$ be a $\Prob_{(X_H(t_1),\dots,X_H(t_d))}$-continuity set, i.e.~a set satisfying
\begin{equation}\label{eq:continuityset}
\Prob{\left({\left(X_H(t_1),\dots,X_H(t_d)\right)} \in \partial B\right)}=\Probb{\left({\left(B_H(t_1),\dots,B_H(t_d)\right)} \in \partial B\right)} = 0,
\end{equation}
where $\partial B$ denotes the boundary of the set $B$ and $\Probb$ is the probability measure defined in \eqref{eq:measureQ}. By the Portmanteau-Theorem for metric spaces, see \cite[Theorem~2.1]{Billingsley1968}, the assertion follows as soon as we have shown
\begin{equation*}
\lim_{T \to \infty} \Prob\big({\left(X_{H,T}(t_1),\dots,X_{H,T}(t_d)\right)} \in B\big)=\Prob\big({\left(X_{H}(t_1),\dots,X_{H}(t_d)\right)} \in B\big).
\end{equation*}

\emph{Step 1:} We use the property of time reversal $(B_H(t))_{t \in {[0,T]}} \id (B_H(T)-B_H(T-t))_{t \in {[0,T]}}$ (second equation) to get
\begin{align*}
&I(T) \cdot \Prob\big({\left(X_{H,T}(t_1),\dots,X_{H,T}(t_d)\right)} \in B\big) \notag
\\&=\E{\left[\I_{\left(T^{-H} B_H(t_i T)\right)_{i=1,\dots,d} \in B} \ \left( \int_0^T \exp(-B_H(s)) \dx s\right)^{-1}\right]} \notag
\\&=\E{\left[\I_{\big(T^{-H} (B_H(T)-B_H(T (1-t_i)))\big)_{i=1,\dots,d} \in B} \cdot \frac {\exp(B_H(T))} {\int_0^T \exp(B_H(T-s)) \dx s}\right]} \notag
\\&=\E{\left[\I_{\big(T^{-H} (B_H(T)-B_H(T (1-t_i)))\big)_{i=1,\dots,d} \in B} \cdot \frac {\exp(B_H(T))} {\int_0^T \exp(B_H(t)) \dx t}\right]} \notag
\\&=\E{\left[\I_{\big(T^{-H} (B_H(T)-B_H(T (1-t_i)))\big)_{i=1,\dots,d} \in B} \cdot \frac {\dx} {\dx T} \log\!{\left(\int_0^T \exp(B_H(t)) \dx t\right)}\right]}.
\end{align*}

\emph{Step 2:} Showing that the derivative of the indicator vanishes.
\\Note that the measure $\Probb$ is equivalent to $\Prob,$ as the density $\frac {\mathrm{d} \Probb} {\mathrm{d} \Prob}=\frac {B_H(1) - M_H(1)} {\E[-M_H(1)]}$ is a.s.~strictly positive. Thus, assumption \eqref{eq:continuityset} implies
\begin{equation*}
\Prob{\left(\big(T^{-H} (B_H(T)-B_H(T (1-t_i)))\big)_{i=1,\dots,d} \in \partial B\right)}=\Prob{\left({\left(B_H(t_1),\dots,B_H(t_d)\right)} \in \partial B\right)}=0
\end{equation*}
for every $T > 0$ by using time reversal backwards and self-similarity. 
\\So, if $\big(T^{-H} (B_H(T)-B_H(T (1-t_i)))\big)_{i=1,\dots,d} \in B$ for some fixed $T>0,$ then a.s.~also $\big(T^{-H} (B_H(T)-B_H(T (1-t_i)))\big)_{i=1,\dots,d} \in B^{\mathrm{o}}$ holds, where $B^{\mathrm{o}}$ denotes the interior of $B.$ Due to the continuity of $B_H$ and the fact that $B^{\mathrm{o}}$ is an open set, this implies a.s.~the existence of some $h_0>0$ s.t.
\begin{equation*}
\big((T+h)^{-H} (B_H(T+h)-B_H((T+h) (1-t_i)))\big)_{i=1,\dots,d} \in B^{\mathrm{o}} \subseteq B
\end{equation*}
for all $|h| \le h_0.$ Analogously, if $\big(T^{-H} (B_H(T)-B_H(T (1-t_i)))\big)_{i=1,\dots,d} \notin B,$ also
\begin{equation*}
\big((T+h)^{-H} (B_H(T+h)-B_H((T+h) (1-t_i)))\big)_{i=1,\dots,d} \notin B
\end{equation*}
a.s.~for $|h|$ small enough, since the boundary of the complement of a set is equal to the boundary of the set itself. This shows that a.s.
\begin{align*}
&\frac {\dx} {\dx T} \,\I_{\big(T^{-H} (B_H(T)-B_H(T (1-t_i)))\big)_{i=1,\dots,d} \in B}
\\&=\lim_{h \to 0} \frac {\I_{\big((T+h)^{-H} (B_H(T+h)-B_H((T+h) (1-t_i)))\big)_{i=1,\dots,d} \in B}-\I_{\big(T^{-H} (B_H(T)-B_H(T (1-t_i)))\big)_{i=1,\dots,d} \in B}} {h} = 0.
\end{align*}

\emph{Step 3:} Putting the first two steps together, we get
\begin{align*}
&I(T) \cdot \Prob\big({\left(X_{H,T}(t_1),\dots,X_{H,T}(t_d)\right)} \in B\big)
\\&=\E{\left[\frac {\dx} {\dx T} \,\I_{\big(T^{-H} (B_H(T)-B_H(T (1-t_i)))\big)_{i=1,\dots,d} \in B} \,\log\!{\left(\int_0^T \exp(B_H(t)) \dx t\right)}\right]}
\\&=\frac {\dx} {\dx T} \,\E{\left[\I_{\big(T^{-H} (B_H(T)-B_H(T (1-t_i)))\big)_{i=1,\dots,d} \in B} \,\log\!{\left(\int_0^T \exp(B_H(t)) \dx t\right)}\right]},
\end{align*}
where we could interchange differential and expectation, since the indicator is uniformly bounded by $1$ and the derivative $T \mapsto \left( \int_0^T \exp(B_H(t)-B_H(T)) \dx s\right)^{-1}$ of the remaining term is majorized (in some neighbourhood of any fixed $T$) by $C_1 \exp\!{\left(C_2 \max_{t \in {[0,1]}} |B_H(t)|\right)}$ for appropriate constants $C_1, C_2.$ Using the substitution $t=sT$ and self-similarity, we conclude
\begin{align}
&I(T) \cdot \Prob\big({\left(X_{H,T}(t_1),\dots,X_{H,T}(t_d)\right)} \in B\big) \notag
\\&=\frac {\dx} {\dx T} \,\E{\left[\I_{\big(T^{-H} (B_H(T)-B_H(T (1-t_i)))\big)_{i=1,\dots,d} \in B} \,\log\!{\left(\int_0^1 \exp(B_H(sT)) \ T \dx s\right)}\right]} \notag
\\&=\frac {\dx} {\dx T} \,\E{\left[\I_{(B_H(1)-B_H(1-t_i))_{i=1,\dots,d} \in B} \left(\log T + \log\!{\left(\int_0^1 \exp{\left(T^H B_H(s)\right)} \dx s\right)}\right)\right]} \notag
\\&=C \,T^{-1} + \E{\left[\frac {\dx} {\dx T} \,\I_{(B_H(1)-B_H(1-t_i))_{i=1,\dots,d} \in B} \,\log\!{\left(\int_0^1 \exp{\left(T^H B_H(s)\right)} \dx s\right)}\right]} \notag
\\&=C \,T^{-1} + H T^{H-1} \,\E{\left[\I_{(B_H(1)-B_H(1-t_i))_{i=1,\dots,d} \in B} \cdot \frac {\int_0^1 B_H(s) \exp{\left(T^H B_H(s)\right)} \dx s} {\int_0^1 \exp{\left(T^H B_H(s)\right)} \dx s}\right]} \label{eq:lastexpressionfdd}
\end{align}
with $C:=\Prob{\left({\left(B_H(t_1),\dots,B_H(t_d)\right)} \in B\right)}.$ Again, we could interchange differential and expectation, since the indicator (which does not depend on $T,$ this time) is uniformly bounded by $1$ and the derivative $T \mapsto H T^{H-1} \,\frac {\int_0^1 B_H(s) \exp{\left(T^H B_H(s)\right)} \dx s} {\int_0^1 \exp{\left(T^H B_H(s)\right)} \dx s}$ of the remaining term can be majorized (in some neighbourhood of any fixed $T$) by a constant multiple of $\max_{t \in {[0,1]}} |B_H(t)|.$ 
\\Note that, by L'Hôpital's rule and the Laplace principle, 
\begin{equation*}
\lim_{r \to \infty} \frac {\int_0^1 f(t) \exp(r f(t)) \,\mathrm{d} t} {\int_0^1 \exp(rf(t)) \,\mathrm{d} t} = \lim_{r \to \infty} r^{-1} \log\!{\left(\int_0^1 \exp(r f(t)) \,\mathrm{d} t\right)} = \max_{t \in {[0,1]}} f(t)
\end{equation*}
holds for any continuous function $f \colon {[0,1]} \to \R.$ Thus, the fraction inside the expectation in \eqref{eq:lastexpressionfdd} converges a.s.~to $S_H(1):=\max_{t \in {[0,1]}} B_H(t).$ Further, it is majorized by $\max_{t \in {[0,1]}} |B_H(t)|.$ Therefore, dominated convergence gives the representation
\begin{align*}
&I(T) \cdot \Prob\big({\left(X_{H,T}(t_1),\dots,X_{H,T}(t_d)\right)} \in B\big)
\\&=C \,T^{-1} + H T^{H-1} \left(\E{\left[\I_{(B_H(1)-B_H(1-t_i))_{i=1,\dots,d} \in B} \,S_H(1)\right]} + o(1)\right)
\end{align*}
for $T \to \infty.$ Furthermore, we have $I(T)=H T^{H-1} \big(\E[S_H(1)] + o(1)\big)$ for $T \to \infty$ by \cite[Statement~1]{Molchan1999}, so together, we get
\begin{align*}
&\Prob\big({\left(X_{H,T}(t_1),\dots,X_{H,T}(t_d)\right)} \in B\big)
\\&=\frac {C \,T^{-1}} {H T^{H-1} \big(\E[S_H(1)] + o(1)\big)} + \frac {\E{\left[\I_{(B_H(1)-B_H(1-t_i))_{i=1,\dots,d} \in B} \,S_H(1)\right]} + o(1)} {\E[S_H(1)] + o(1)}
\\&\to 0 + \E{\left[\I_{(B_H(1)-B_H(1-t_i))_{i=1,\dots,d} \in B} \cdot \frac {S_H(1)} {\E[S_H(1)]}\right]}
\end{align*}
for $T \to \infty.$ Using again the property of time reversal, the limit is given by
\begin{align*}
&\E{\left[\I_{(B_H(t_i))_{i=1,\dots,d} \in B} \cdot \frac {\max_{t \in {[0,1]}} \big(B_H(1) - B_H(1-t)\big)} {\E[S_H(1)]}\right]}
\\&=\Probb{\left({\left(B_H(t_1),\dots,B_H(t_d)\right)} \in B\right)}=\Prob\big({\left(X_{H}(t_1),\dots,X_{H}(t_d)\right)} \in B\big). \qedhere
\end{align*}
\end{proof}
In order to finish the proof of Theorem \ref{thm:main}, we show tightness in the following lemma:
\begin{lemma}
In the setting of Theorem \ref{thm:main}, the family $\big(\Prob(X_{H,T} \in \,\cdot ~)\big)_{T \ge 1}$ of distributions on $(C{[0,1]}, \lVert \cdot \rVert_{\infty})$ is tight. 
\end{lemma}
\begin{proof}
By \cite[Theorem~8.2]{Billingsley1968}, it suffices to show that
\begin{equation}\label{eq:criterionTightness}
\lim_{\delta \downarrow 0} \,\limsup_{T \to \infty} \,\Prob{\left(w_{X_{H,T}}(\delta) \ge \eps\right)} = 0
\end{equation}
for all $\eps > 0,$ where $w_f(\delta):=\sup_{|t-s|< \delta} |f(s)-f(t)|, \ \delta \in {(0,1]},$ is the modulus of continuity of a function $f \in C{[0,1]}.$ 

\emph{Step 1:} Let $\eps > 0, \ \delta \in {(0,1]}$ and $T \ge 1.$ By doing the same as in Step 1 of the proof of Lemma \ref{lem:fdd}, we get
\begin{align*}
&I(T) \cdot \Prob{\left(w_{X_{H,T}}(\delta) \ge \eps\right)}
\\&=\E{\left[\I_{\sup_{|t-s|<\delta} T^{-H} |B_H(tT)-B_H(sT)| \ge \eps} \ \left( \int_0^T \exp(-B_H(s)) \dx s\right)^{-1}\right]} \notag
\\&=\E{\left[\I_{\sup_{|t-s|<\delta} T^{-H} |B_H(T)-B_H((1-t)T)-B_H(T)+B_H((1-s)T)| \ge \eps} \cdot \frac {\exp(B_H(T))} {\int_0^T \exp(B_H(T-s)) \dx s}\right]} \notag
\\&=\E{\left[\I_{\sup_{|t-s|<\delta} T^{-H} |B_H(tT)-B_H(sT)| \ge \eps} \cdot \frac {\exp(B_H(T))} {\int_0^T \exp(B_H(t)) \dx t}\right]} \notag
\\&=\E{\left[\I_{\sup_{|t-s|<\delta} T^{-H} |B_H(tT)-B_H(sT)| \ge \eps} \cdot \frac {\dx} {\dx T} \log\!{\left(\int_0^T \exp(B_H(t)) \dx t\right)}\right]}.
\end{align*}

\emph{Step 2:} Showing that the derivative of the indicator vanishes.
\\We need to justify that
\begin{equation}
\frac {\dx} {\dx T} \,\I_{\sup_{\substack{t,s \in {[0,1]}\colon \\ |t-s|<\delta}} T^{-H} |B_H(tT)-B_H(sT)| \,\ge \,\eps}=\frac {\dx} {\dx T} \,\I_{\sup_{\substack{t,s \in {[0,T]}\colon \\ |t-s|<\delta T}} |B_H(t)-B_H(s)| \,\ge \,\eps T^H}=0 \label{eq:derivativeIndicatorTightness}
\end{equation}
holds a.s.: Obviously, the event $\left\{\sup_{t,s \in {[0,T]}\colon |t-s|<\delta T} |B_H(t)-B_H(s)| = \eps T^H\right\}$ has probability zero. Let us first assume $\sup_{t,s \in {[0,T]}\colon |t-s|<\delta T} |B_H(t)-B_H(s)| < \eps T^H$ for some fixed $T \ge 1.$ As $B_H$ is even a.s.~uniformly continuous on the finite interval ${[0,2T]},$ there exists some $\eta < T$ s.t. $|B_H(t)-B_H(s)| < \frac {1} {2} \left(\eps T^H - \sup_{t,s \in {[0,T]}\colon |t-s|<\delta T} |B_H(t)-B_H(s)|\right)$ for all $t,s \in {[0,2T]}$ with $|t-s| < \eta.$ Then, also
\begin{equation*}
\sup_{t,s \in {[0,T+h]}\colon |t-s|<\delta (T+h)} |B_H(t)-B_H(s)| < \eps T^H < \eps (T+h)^H
\end{equation*}
for all $0< h < \eta,$ as the triangle inequality gives
\begin{align*}
&\sup_{t,s \in {[0,T+h]}\colon |t-s|<\delta (T+h)} |B_H(t)-B_H(s)|
\\&\le \max\bigg\{\sup_{s,t \in {(T,T+h]}} |B_H(t)-B_H(s)|, 
\\&\sup_{s < t \in {[0,T]}\colon t-s < \delta (T+h)} \big(|B_H(t) - B_H(s+\delta T)| + |B_H(s+\delta T) - B_H(s)|\big),
\\&\sup_{\substack{s \in {[0,T]}, \ t \in {(T,T+h]}\colon \\ t-s < \delta (T+h)}} \big(|B_H(t) - B_H(T)| + |B_H(T)-B_H(s+\delta T)| + |B_H(s+\delta T) - B_H(s)|\big)\bigg\}.
\end{align*}
For $h<0,$ trivially,
\begin{equation*}
\sup_{t,s \in {[0,T+h]}\colon |t-s|<\delta (T+h)} |B_H(t)-B_H(s)| \le \sup_{t,s \in {[0,T]}\colon |t-s|<\delta T} |B_H(t)-B_H(s)| < \eps T^H
\end{equation*}
and for $|h|$ small enough, even $\sup_{t,s \in {[0,T]}\colon |t-s|<\delta T} |B_H(t)-B_H(s)| < \eps (T+h)^H$ holds, as the term on the left hand side is independent of $h.$ Thus, we have in this case
\begin{equation}\label{eq:differenceindicators}
\I_{\sup_{\substack{t,s \in {[0,T]}\colon \\ |t-s|<\delta T}} |B_H(t)-B_H(s)| \,\ge \,\eps T^H}=\I_{\sup_{\substack{t,s \in {[0,T+h]}\colon \\ |t-s|<\delta (T+h)}} |B_H(t)-B_H(s)| \,\ge \,\eps (T+h)^H}
\end{equation}
for $|h|$ small enough. 
\\Conversely, if $\sup_{t,s \in {[0,T]}\colon |t-s|<\delta T} |B_H(t)-B_H(s)| > \eps T^H,$ there exists a.s.~an $\eta < T$ s.t. $|B_H(t)-B_H(s)| < \frac {1} {2} \left(\sup_{t,s \in {[0,T]}\colon |t-s|<\delta T} |B_H(t)-B_H(s)| - \eps T^H\right)$ for all $t,s \in {[0,2T]}$ with $|t-s| < \eta.$ Then, again by the triangle inequality, also
\begin{equation*}
\sup_{t,s \in {[0,T+h]}\colon |t-s|<\delta (T+h)} |B_H(t)-B_H(s)| > \eps T^H > \eps (T+h)^H
\end{equation*}
for all $-\eta < h < 0.$ Analogously to above, for $h>0$ small enough, 
\begin{equation*}
\sup_{t,s \in {[0,T+h]}\colon |t-s|<\delta (T+h)} |B_H(t)-B_H(s)| \ge \sup_{t,s \in {[0,T]}\colon |t-s|<\delta T} |B_H(t)-B_H(s)| > \eps (T+h)^H,
\end{equation*}
so also in this case, \eqref{eq:differenceindicators} holds eventually for $h \to 0.$ Together, we conclude \eqref{eq:derivativeIndicatorTightness}. 

\emph{Step 3:} Following the proof of Lemma \ref{lem:fdd} until \eqref{eq:lastexpressionfdd}, we get
\begin{align}
&I(T) \cdot \Prob{\left(w_{X_{H,T}}(\delta) \ge \eps\right)} \notag
\\=\ &T^{-1} \,\Prob{\left(\sup_{t,s \in {[0,1]}\colon |t-s|<\delta} |B_H(t)-B_H(s)| \ge \eps\right)} \notag
\\&+ H T^{H-1} \,\E{\left[\frac {\int_0^1 B_H(s) \exp{\left(T^H B_H(s)\right)} \dx s} {\int_0^1 \exp{\left(T^H B_H(s)\right)} \dx s} \ \I_{\sup_{t,s \in {[0,1]}\colon |t-s|<\delta} |B_H(t)-B_H(s)| \ge \eps}\right]}. \label{eq:lastexpressionTightness}
\end{align}
As before, we set $S_H(1):=\max_{t \in {[0,1]}} B_H(t)$ and use $I(T)=H T^{H-1} \big(\E[S_H(1)] + o(1)\big)$ for $T \to \infty.$ This implies the existence of some $T_0 \ge 1$ s.t. $I(T) \ge \frac {1} {2} H T^{H-1} \E[S_H(1)]$ for all $T \ge T_0.$ Since the indicator inside the second expectation of \eqref{eq:lastexpressionTightness} is non-negative and the fraction is bounded by $S_H(1),$ we get
\begin{align*}
\Prob{\left(w_{X_{H,T}}(\delta) \ge \eps\right)} \le \ &\frac {2} {H \,\E[S_H(1)]} \,T^{-H} \,\Prob{\left(\sup_{t,s \in {[0,1]}\colon |t-s|<\delta} |B_H(t)-B_H(s)| \ge \eps\right)} 
\\&+ \frac {2} {\E[S_H(1)]} \,\E{\left[S_H(1) \ \I_{\sup_{t,s \in {[0,1]}\colon |t-s|<\delta} |B_H(t)-B_H(s)| \ge \eps}\right]}
\\\le \ &\frac {2} {H \,\E[S_H(1)]} \,\Prob{\left(\sup_{t,s \in {[0,1]}\colon |t-s|<\delta} |B_H(t)-B_H(s)| \ge \eps\right)} 
\\&+ \frac {2} {\E[S_H(1)]} \,\lVert S_H(1) \rVert_2 \,\Prob{\left(\sup_{t,s \in {[0,1]}\colon |t-s|<\delta} |B_H(t)-B_H(s)| \ge \eps\right)}^{\frac {1} {2}}
\\\le \ &C \ \Prob{\left(\sup_{t,s \in {[0,1]}\colon |t-s|<\delta} |B_H(t)-B_H(s)| \ge \eps\right)}^{\frac {1} {2}}
\end{align*}
for all $T \ge T_0$ and $C:=\frac {2} {H \,\E[S_H(1)]} + \frac {2} {\E[S_H(1)]} \,\lVert S_H(1) \rVert_2,$ where we used the Cauchy-Schwarz inequality as well as $T \ge 1$ in the second and $x \le \sqrt{x}$ for $x \in {[0,1]}$ in the third step. \\So in order to conclude \eqref{eq:criterionTightness} and thus the assertion, it suffices to show 
\begin{equation*}
\lim_{\delta \downarrow 0} w_{B_H}(\delta) = \lim_{\delta \downarrow 0} \,\Prob{\left(\sup_{t,s \in {[0,1]}\colon |t-s|<\delta} |B_H(t)-B_H(s)| \ge \eps\right)} = 0.
\end{equation*}
But due to the continuity of $\Prob$ and the a.s.~continuity of $B_H,$ we have
\begin{align*}
&\lim_{\delta \downarrow 0} \,\Prob{\left(\sup_{t,s \in {[0,1]}\colon |t-s|<\delta} |B_H(t)-B_H(s)| \ge \eps\right)} 
\\&= \Prob{\left(\bigcap_{\delta > 0, \delta \in \Q} \left\{\sup_{t,s \in {[0,1]}\colon |t-s|<\delta} |B_H(t)-B_H(s)| \ge \eps\right\}\right)} = 0. \qedhere
\end{align*}
\end{proof}

\subsection{Proof of Proposition \texorpdfstring{\ref{prop:BrownianCase}}{2.2}}

\begin{proof}
In the following, we abbreviate $(B_{1/2}(t))_{t \in [0,1]}$ to $(B(t))_{t \in [0,1]}$ and $(M_{1/2}(t))_{t \in [0,1]}$ to $(M(t))_{t \in [0,1]}$ to be consistent with the notation in the proposition. 
The key observation in the proof of the proposition is that, by Girsanov's Theorem, an absolutely continuous change of measure corresponds to a change of drift.
The main task is to show that the process $(\tilde{B}(t))_{t \in [0,1]}$ given by
\begin{equation*}
\tilde{B}(t) = B(t) - \int_0^t c(s,B(s)-M(s)) \dx s, \ t \in {[0,1]},
\end{equation*}
is a Brownian motion under $\Probb.$
Then, recalling that under $\Probb,$ the process $(B(t))_{t \in [0,1]}$ has the same distribution as the process $(X(t))_{t \in [0,1]}$ under $\Prob,$ the claim follows.  

Since $\Prob$ is equivalent to $\Probb,$ by \cite[Theorem IV.38.5]{Rogers2000}, a previsible process $(c(t))_{t \in [0,1]}$ exists such that 
\begin{equation*}
Z(t) := \left.\frac{\dx \Probb}{\dx \Prob} \right|_{\mathcal{F}_t} = \exp\!{\left( \int_0^t c(s) \dx B(s) - \frac{1}{2} \int_0^t c(s)^2 \dx s \right)}
\end{equation*}
and $(\tilde{B}(t))_{t \in [0,1]}$ is a Brownian motion under $\Probb.$ Here, $(\mathcal{F}_t)_{t \in [0,1]}$ denotes the natural filtration of $(B(t))_{t \in [0,1]}.$
Thus, it remains to derive the explicit expression of $c(t)$ and to show that it coincides with $c(t,B(t)-M(t)).$ For this purpose, we note that $(Z(t))_{t \in [0,1]}$ is the unique solution of the SDE 
\begin{equation}
\label{eq:ZsolutionSDE}
Z(t) = 1 + \int_0^t Z(s) c(s) \dx B(s),
\end{equation}
see e.g.~\cite[Example 3.9]{KaratzasShreve}.
Further, by \cite[Theorem~1]{Shiryaev2004}, $M(1)$ has the stochastic integral representation 
\begin{equation*}
M(1) = \E[M(1)] - 2\int_0^1 (\Phi_{1-s}(B(s)-M(s))-1) \dx B(s).
\end{equation*}
We thus obtain, together with \eqref{eq:measureQ} and the fact that $(Z(t))_{t \in [0,1]}$ is a $\Prob$-martingale,
\begin{align*}
Z(t) 
&= 
\E[ Z(1) \mid \mathcal{F}_t ] \\
&=
\E {\left[ \left.\frac{B(1)-M(1)}{\E[-M(1)]} \ \right| \mathcal{F}_t \right]} \\
&=
\E {\left[ \left.1 + \frac{\int_0^1 (2 \,\Phi_{1-s}(B(s)-M(s))-1) \dx B(s)}{\E[-M(1)]} \ \right| \mathcal{F}_t \right]} \\
&=
1 + \frac{\int_0^t (2 \,\Phi_{1-s}(B(s)-M(s))-1) \dx B(s)}{\E[-M(1)]}, \quad t \in [0,1].
\end{align*}
Consequently, we can conclude from \eqref{eq:ZsolutionSDE} that
\begin{equation}
\label{eq:Zc}
Z(s)c(s) = \frac{2 \,\Phi_{1-s}(B(s)-M(s))-1}{\E[-M(1)]}, \quad s \in [0,1].
\end{equation}
Now, using again that $(Z(t))_{t \in [0,1]}$ is a $\Prob$-martingale, and thus, $Z(t) = \E{\left[\left.\frac{ B(1)-M(1)  }{\E[-M(1)]} \ \right| \mathcal{F}_t \right]},$ the computation of $c(t)$ reduces to a computation of the conditional expectation $\E[ B(1)-M(1) \mid \mathcal{F}_t ].$
To do this, we split the process at time $t$ and define $B^{(t)} := B(1)-B(t)$ and $M^{(t)} := \min_{s \in [t,1]} B(s)-B(t),$ respectively, to simplify notation. 
Then, we have
\begin{align}
\E[ B(1)-M(1) \mid \mathcal{F}_t ] \notag
={}&
\E{\left[ \I_{|M^{(t)}| > B(t)-M(t)} (B^{(t)}-M^{(t)}) \mid \mathcal{F}_t \right]} \notag \\
&+
\E{\left[ \I_{|M^{(t)}| \leq B(t)-M(t)} (B(t)+B^{(t)}-M(t)) \mid \mathcal{F}_t \right]}\notag \\
={}&
\E{\left[ \I_{|M^{(t)}| > B(t)-M(t)} | M^{(t)} | \mid \mathcal{F}_t \right]}\notag \\
&+
\E{\left[ \I_{|M^{(t)}| \leq B(t)-M(t)} (B(t)-M(t)) \mid \mathcal{F}_t \right]}\notag \\
&+
\E[B^{(t)}]\notag \\
={}&
\int_0^\infty \Prob{\left( \I_{|M^{(t)}| > B(t)-M(t)} | M^{(t)} | > s \mid \mathcal{F}_t\right)} \dx s\notag \\
&+
(B(t)-M(t)) \cdot \Prob{\left( |M^{(t)}| \leq B(t)-M(t) \mid \mathcal{F}_t\right)} + 0 \notag\\
={}&
(B(t)-M(t)) \cdot \Prob{\left( |M^{(t)}| > B(t)-M(t) \mid \mathcal{F}_t\right)}\notag\\
&
+\int_{B(t)-M(t)}^\infty \Prob( | M^{(t)} | > s ) \dx s \notag\\
&+
(B(t)-M(t)) \cdot \Prob{\left( |M^{(t)}| \leq B(t)-M(t) \mid \mathcal{F}_t\right)}\notag \\
={}&
(B(t)-M(t)) + \int_{B(t)-M(t)}^\infty \Prob( | M^{(t)} | > s ) \dx s. \label{eq:condExpect}
\end{align}
Using $\Prob( | M^{(t)} | > s ) = 2 \,(1-\Phi_{1-t}(s)),$ which is a simple conclusion from the reflection principle, the claim follows from \eqref{eq:Zc} and \eqref{eq:condExpect}. 
\end{proof}

\noindent\textbf{Acknowledgement.} This work was supported by Deutsche
Forschungsgemeinschaft (DFG grant AU370/5). We would like to thank Hugo Panzo
(Technion) for pointing out reference \cite{MansuyYor2008} to us. Further, we would like to thank two anonymous referees for their valuable suggestions. 

\nocite{*}
\bibliographystyle{plain}

\end{document}